\documentclass[letterpaper, 10 pt, journal, twoside]{IEEEtran}

\IEEEoverridecommandlockouts

\usepackage[space]{cite}

\usepackage{amsmath}
\usepackage{amsfonts}
\usepackage{amssymb}
\usepackage{amsthm}
\usepackage{mathrsfs}
\usepackage{xspace}
\usepackage{xparse}

\usepackage{enumitem}

\usepackage{algorithm}
\usepackage{algpseudocode}
\usepackage{setspace}

\usepackage{tabularx}
\usepackage{booktabs}
\usepackage{multirow}

\usepackage{tikz}
\usetikzlibrary{positioning}

\usepackage{epstopdf}
\usepackage{lipsum}
\usepackage{color, soul}

\newtheorem{theorem}{Theorem}
\newtheorem{lemma}[theorem]{Lemma}
\newtheorem{corollary}[theorem]{Corollary}
\newtheorem{definition}{Definition}
\newtheorem{problem}{Problem}
\newtheorem{proposition}[theorem]{Proposition}

\newcommand{\propositionref}[1]   {Proposition \ref{#1}\xspace}
\newcommand{\lemmaref}[1]         {Lemma \ref{#1}\xspace}

\newcommand{\algorithmref}[1]     {Algorithm \ref{#1}\xspace}

\renewcommand{\Pr}{\mathop{\rm Pr}\nolimits}
\newcommand{\overbar}[1]{\mkern 1.5mu\overline{\mkern-1.5mu#1\mkern-1.5mu}\mkern 1.5mu}

\title{\LARGE \bf
Model-Free Stochastic Reachability Using Kernel Distribution Embeddings
}
\author{Adam J. Thorpe, Meeko M. K. Oishi
\thanks{This material is based upon work supported by the National Science Foundation
    under NSF Grant Number CMMI-1254990, IIS-1528047, and CNS-1836900,
    and by the Air Force Research Lab under
    AFRL Cooperative Agreement FA9453-18-2-0022, Agile Manufacturing for High Value, Low Volume Production.
    Any opinions, findings, and conclusions or recommendations expressed in this
    material are those of the authors and do not necessarily reflect the views
    of the National Science Foundation. \newline\indent
    A. Thorpe and M. Oishi are with Electrical \& Computer Eng., University of New Mexico, Abq., NM. Email: {\tt\{ajthor,oishi\}@unm.edu}.
    }
}

\begin{document}
\maketitle

\begin{abstract}
We present a data-driven solution to the terminal-hitting stochastic reachability problem for a Markov control process.
We employ a nonparametric representation of the stochastic kernel as a conditional distribution embedding within a reproducing kernel Hilbert space (RKHS).
This representation avoids intractable integrals in the dynamic recursion of the stochastic reachability problem
since the expectations can be calculated as an inner product within the RKHS.
We demonstrate this approach on a high-dimensional chain of integrators and on Clohessy-Wiltshire-Hill dynamics.
\end{abstract}



\section{Introduction}

Stochastic reachability is an established verification tool to assure that a system will reach a desired state without violating predefined safety constraints
with at least a desired likelihood.
%
%
%
The solution to the stochastic reachability problem is based on dynamic programming
\cite{summers2010verification, abate2008probabilistic},
which poses significant computational challenges.
Optimization-based solutions have garnered modest computational tractability
via chance constraints \cite{lesser2013stochastic, vinod2019piecewise}, sampling
methods \cite{sartipizadeh2018voronoi, vinod2018multiple, vinod2018stochastic}, and convex optimization with Fourier
transforms \cite{vinod2018scalable, vinod2017scalable}, but are limited to linear
dynamical systems with Gaussian or log-concave disturbances.
Methods using approximate dynamic programming
\cite{kariotoglou2013approximate} and
particle filtering \cite{manganini2015policy, lesser2013stochastic},
are applicable to classes of nonlinear systems, but have only been demonstrated on systems with
moderate dimensionality.

Further, for many dynamical systems, presumption of accurate knowledge of dynamics and uncertainty is unrealistic. Historically, such uncertainty is handled through approximations and introduction of error terms that bound unknown elements \cite{mitchell2005time, fisac2015reach}.  With the rapid increase in the use of learning elements that are resistant to traditional models for control and formal methods, as well as the involvement of humans, such approaches may either be overly conservative or even simply inaccurate. For example, human inputs may be highly heterogeneous, may not follow a known distribution, and are often data-driven processes that may be biased when analyzed through sampling methods.

\emph{We propose to use conditional distribution embeddings within a reproducing kernel Hilbert space
(RKHS) to solve the stochastic reachability problem.}
As a tool for stochastic reachability, kernel methods offer significant advantages over state-of-the-art: 1) they are model-free, so can accommodate data-driven stochastic processes and nonlinear dynamics, 2) they are convergent in probability, and 3) they do not suffer from the curse of dimensionality
\cite{bellman2015applied}
or from issues of numerical
quadrature that plague optimization based approaches.
The primary computational challenge arises in the inversion of a matrix that scales with the number of data points, leading to computational complexity that is exponential in the size of the data.
%

Kernel methods are an established learning technique
\cite{scholkopf2001learning, shawe2004kernel, aronszajn1950theory}, and have
recently been applied to dynamic programming problems
with additive cost function and infinite time horizon
\cite{grunewalder2012modelling}.
Kernel methods broadly enable nonparametric inference using
kernel embeddings of distributions.  They can
capture the features of arbitrary statistical
distributions in a data-driven fashion \cite{song2013kernel,
smola2007hilbert}.
Kernel methods
do not suffer from biases or prior assumptions on the system model, and
are computationally efficient because they are convergent and
non-iterative \cite{scholkopf2001learning}.
These techniques have also been applied to several problems in dynamical
systems, including controller synthesis
\cite{lever2015modelling}, partially-observable systems \cite{nishiyama2012hilbert}, and estimation
of graphical models \cite{song2009hilbert}.

{\em The main contribution of this paper is the application of conditional distribution embedding
to compute the stochastic reachability probability measure, to enable model-free verification without
invoking a statistical approach.}
This is particularly relevant for systems with
black-box elements, such as autonomous or human-in-the-loop systems, which have
traditionally been resistant to formal verification techniques.
We tailor the approach in \cite{grunewalder2012modelling} to
accommodate the multiplicative cost and finite horizon associated with the stochastic reachability
dynamic program.
We apply kernel methods to compute the stochastic reachability
probability measure 
by representing
the stochastic kernel
as a conditional distribution embedding within an RKHS.
%
%

The paper organization is as follows.
Section \ref{section: problem formulation} formulates the problem.
%
Section \ref{section: kernel mean embeddings} applies
conditional distribution embeddings to compute the stochastic reachability probability measure.
%
%
Section \ref{section: numerical results} demonstrates our approach on
three examples: a double integrator to enable validation with a ``truth'' model
via dynamic programming, a 10,000-dimensional integrator to demonstrate scalability, and spacecraft rendezvous and docking.


\section{Problem Formulation}
\label{section: problem formulation}

For sets $\mathcal{A}$ and $\mathcal{B}$,
the set of all elements of $\mathcal{A}$ which are not in $\mathcal{B}$
is denoted as $\mathcal{A} \backslash \mathcal{B}$.
We denote the indicator function
as $\boldsymbol{1}_{\mathcal{A}}(x) = 1$ if $x \in \mathcal{A}$
and $\boldsymbol{1}_{\mathcal{A}}(x) = 0$ if $x \notin \mathcal{A}$.



Let $\Omega$ denote a sample space
and $\mathcal{F}(\Omega)$ denote the $\sigma$-algebra relative to $\Omega$.
A probability measure $\Pr$ assigned to
the measurable space $(\Omega, \mathcal{F}(\Omega))$
is defined as the probability space $(\Omega, \mathcal{F}(\Omega), \Pr)$.
When $\Omega \equiv \Re$,
the $\sigma$-algebra of $\Omega$ is denoted as $\mathscr{B}(\Omega)$,
and is the Borel $\sigma$-algebra associated with $\Omega$.
A random variable $\boldsymbol{x}$ is a measurable function
on the probability space $(\Omega, \mathcal{F}(\Omega), \Pr)$.
A random vector $\boldsymbol{x} = [\boldsymbol{x}_{1}, \ldots, \boldsymbol{x}_{n}]^{\top}$
of $n$ random variables $\lbrace \boldsymbol{x}_{i} \rbrace_{i=1}^{n}$, each measurable functions on the probability space $(\Omega, \mathcal{F}(\Omega), \Pr)$, is defined on the induced probability space
$(\Omega^{n}, \mathcal{F}(\Omega^{n}), \Pr_{\boldsymbol{x}})$,
where $\Pr_{\boldsymbol{x}}$ is the induced probability measure.
A stochastic process is defined as a sequence of random vectors
$\lbrace \boldsymbol{x}_{k} : k \in [0, N]\rbrace$, $N \in \mathbb{N}$,
where $\boldsymbol{x}_{k}$ are defined on the probability space
$(\Omega^{n}, \mathcal{F}(\Omega^{n}), \Pr_{\boldsymbol{x}})$.
See \cite{billingsley2008probability, chow2012probability} for more details.

The expectation operator is denoted as $\mathbb{E}[\,\cdot\,]$,
where for some function $f$,
$\mathbb{E}_{\boldsymbol{x} \sim \Pr_{\boldsymbol{x}}\lbrace\,\cdot\,\rbrace}[f(\boldsymbol{x})]$
denotes the expectation operator
with respect to the probability measure $\Pr_{\boldsymbol{x}}$.


\subsection{Terminal-Hitting Time Problem}
\label{section: system model}

Consider a Markov control process $\mathcal{H}$,
which is defined in \cite{summers2010verification} as a 3-tuple,
\begin{align}
	\mathcal{H} = ( \mathcal{X}, \mathcal{U}, Q )
	\label{eqn: Markov control process}
\end{align}
where $\mathcal{X} \subseteq \Re^{n}$ is the state space,
$\mathcal{U} \subseteq \Re^{m}$ is the control space,
and
$Q : \mathscr{B}(\mathcal{X}) \times{} \mathcal{X} \times{} \mathcal{U} \rightarrow [0, 1]$
is a stochastic kernel,
which is a Borel-measurable function
that maps a probability measure $Q(\,\cdot\, | \,x, u)$
to each $x \in \mathcal{X}$ and $u \in \mathcal{U}$
on the Borel space $(\mathcal{X}, \mathscr{B}(\mathcal{X}))$.
Further, let $\mathcal{X}$ and $\mathcal{U}$ be compact Borel spaces.
The system evolves over a finite horizon $k \in [0, N]$
with inputs chosen according to a Markov policy
\cite{puterman2014markov, bertsekas2004stochastic}, a sequence
$\pi = \lbrace \pi_{0}, \pi_{1}, \ldots, \pi_{N-1}\rbrace$
of universally-measurable maps
$\pi_{k} : \mathcal{X} \rightarrow \mathcal{U}$.
The set of all Markov control policies $\pi$ is denoted as $\mathcal{M}$.


We define $\mathcal{K}, \mathcal{T} \in \mathscr{B}(\mathcal{X})$
as the \emph{safe set} and \emph{target set}, respectively.
We define the
\emph{terminal-hitting time safety probability}
$r_{x_{0}}^{\pi}(\mathcal{K}, \mathcal{T})$
\cite{summers2010verification}
as the probability that a system $\mathcal{H}$
controlled by a policy $\pi \in \mathcal{M}$
will {reach} $\mathcal{T}$ at $k=N$
while {avoiding} $\mathcal{X}\backslash\mathcal{K}$
for all $k \in [0, N-1]$,
given an initial condition $x_{0} \in \mathcal{X}$.
\begin{align}
  r_{x_{0}}^{\pi}(\mathcal{K}, \mathcal{T}) \triangleq
    \Pr_{x_{0}}^{\pi}
    \lbrace
        \boldsymbol{x}_{N} \in \mathcal{T} \wedge{}
        \boldsymbol{x}_{i} \in \mathcal{K}, \forall i \in [0, N-1]
    \rbrace
  \label{eqn: terminal-hitting probability}
\end{align}
Let $V{}_{k}^{\pi} : \mathcal{X} \rightarrow [0, 1]$
be defined via backward recursion as
\begin{align}
  V{}_{N}^{\pi}(x) &=
	\boldsymbol{1}_{\mathcal{T}}(x)
	\label{eqn: terminal-hitting value N} \\
  V{}_{k}^{\pi}(x) &=
	\boldsymbol{1}_{\mathcal{K}}(x)
  \mathbb{E}_{\boldsymbol{y} \sim Q(\,\cdot\,|\,x, \pi_{k}(x))}
  \left[ V{}_{k+1}^{\pi}(\boldsymbol{y}) \right]
  \label{eqn: terminal-hitting value k}
\end{align}
then $V{}_{0}^{\pi}(x) = r_{x_{0}}^{\pi}(\mathcal{K}, \mathcal{T})$
for every $x_{0} \in \mathcal{X}$.

From \cite[Definition~10]{summers2010verification},
a policy $\pi^{*} \in \mathcal{M}$
is the \emph{maximal reach-avoid policy in the terminal sense}
if and only if 
\begin{align}
  r_{x_{0}}^{\pi^{*}}(\mathcal{K}, \mathcal{T}) =
  \sup_{\pi \in \mathcal{M}}
  \lbrace
    r_{x_{0}}^{\pi}(\mathcal{K}, \mathcal{T})
  \rbrace
\end{align}
Let $V{}_{k}^{*} : \mathcal{X} \rightarrow [0, 1]$, $k \in [0, N-1]$
be defined via backward recursion,
initialized with $V{}_{N}^{*}(x) = \boldsymbol{1}_{\mathcal{T}}(x)$, as
\begin{align}
  V{}_{k}^{*}(x) =
  \sup_{u \in \mathcal{U}}
  \left\lbrace
    \boldsymbol{1}_{\mathcal{K}}(x)
    \mathbb{E}_{\boldsymbol{y} \sim Q(\,\cdot\,|\,x,u)}
    \left[ V{}_{k+1}^{*}(\boldsymbol{y}) \right]
  \right\rbrace
  \label{eqn: terminal-hitting optimal value k}
\end{align}
Then, $V{}_{0}^{*}(x) = r_{x_{0}}^{\pi^{*}}(\mathcal{K}, \mathcal{T})$.
%
If $\pi^{*}$
is such that 
\begin{align}
  \pi_{k}^{*}(x) =
  \mathrm{arg}\sup_{u \in \mathcal{U}}
  \left\lbrace
    \boldsymbol{1}_{\mathcal{K}}(x)
    \mathbb{E}_{\boldsymbol{y} \sim Q(\,\cdot \,|\, x, u)}
    \left[ V{}_{k+1}^{*}(\boldsymbol{y}) \right]
  \right\rbrace
  \label{eqn: terminal-hitting optimal control map}
\end{align}
then $\pi^{*}$
is maximal in the terminal sense \cite[Theorem~11]{summers2010verification}.


\subsection{Problem Statement}

Consider a set $\mathcal{S}$ of $M$ samples
of the form $\mathcal{S} = \lbrace (\bar{y}_{i}, \bar{x}_{i}, \bar{u}_{i}) \rbrace_{i = 1}^M$ such that $y_{i}$ 
is drawn i.i.d. from $Q$
according to $\bar{y}_{i} \sim Q(\,\cdot \,|\, \bar{x}_{i}, \bar{u}_{i})$,
and $\bar{u}_{i} = \pi(\bar{x}_{i})$.
We denote sample vectors with a bar to differentiate them from time-indexed vectors.

\begin{problem}
Without direct knowledge of $Q$,
use samples $\mathcal{S}$
to construct a kernel-based approximation
of \eqref{eqn: terminal-hitting value k}
that converges in probability.
\end{problem}

\begin{problem}
Without direct knowledge of $Q$,
use samples $\mathcal{S}$
to construct a kernel-based approximation
of \eqref{eqn: terminal-hitting optimal value k}
that converges in probability,
in order to compute an approximation of the optimal policy
$\pi^{*}$.
\end{problem}

Using samples $\mathcal{S}$, we employ an approach similar to that in \cite{grunewalder2012modelling}, but that is specific to  the
dynamic program associated with the stochastic reachability problem
for high-dimensional, non-Gaussian systems.
The unique computational efficiencies
afforded by
reproducing kernel Hilbert spaces
transforms computation of
\eqref{eqn: terminal-hitting value k} and
\eqref{eqn: terminal-hitting optimal value k}
into simple matrix operations and inner products.


\section{Kernel Distribution Embeddings for Stochastic~Reachabilty}
\label{section: kernel mean embeddings}

For some set $\mathcal{X}$,
let $\mathscr{H}_{\mathcal{X}}$ denote the unique
reproducing kernel Hilbert space
\cite{scholkopf2001learning}
with the positive definite \cite[Definition~4.15]{steinwart2008support} kernel
$K_{\mathcal{X}} : \mathcal{X} \times{} \mathcal{X} \rightarrow \Re$,
which is a Hilbert space of real-valued functions on $\mathcal{X}$
with inner product
$\langle \,\cdot, \cdot\, \rangle_{\mathscr{H}_{\mathcal{X}}}$
and the induced norm
$\lVert x \rVert_{\mathscr{H}_{\mathcal{X}}} = \left(\langle x, x \rangle_{\mathscr{H}_{\mathcal{X}}} \right)^{1/2}$.
A reproducing kernel Hilbert space has two important properties
\cite{aronszajn1950theory}:
\begin{enumerate}
	\item For any $x, x' \in \mathcal{X}$,
				$K_{\mathcal{X}}(x, \cdot\,) : x' \rightarrow K_{\mathcal{X}}(x, x')$
				is an element of $\mathscr{H}_{\mathcal{X}}$.
	\item An element $K_{\mathcal{X}}(x, x')$ of $\mathscr{H}_{\mathcal{X}}$
				satisfies the \emph{reproducing property}
				such that $\forall f \in \mathscr{H}_{\mathcal{X}}$
				and $x \in \mathcal{X}$,
				\begin{align}
					f(x) &=
					\langle
						K_{\mathcal{X}}(x, \cdot), f(\cdot)
					\rangle_{\mathscr{H}_{\mathcal{X}}}
					\label{eqn: reproducing property} \\
					K_{\mathcal{X}}(x, x') &=
	\langle
		K_{\mathcal{X}}(x, \cdot),
		K_{\mathcal{X}}(x', \cdot)
	\rangle_{\mathscr{H}_{\mathcal{X}}}
				\end{align}
\end{enumerate}


This means that 
the evaluation of a function $f \in \mathscr{H}_{\mathcal{X}}$
can be viewed as an inner product
in $\mathscr{H}_{\mathcal{X}}$.
%
Alternatively, an element $K_{\mathcal{X}}(x, \cdot)$
can be viewed as a nonlinear feature map
$\phi : \mathcal{X} \rightarrow \mathscr{H}_{\mathcal{X}}$,
such that
\begin{align}
	K_{\mathcal{X}}(x, x') =
	\langle
		\phi(x),
		\phi(x')
	\rangle_{\mathscr{H}_{\mathcal{X}}}
\end{align}
Because constructing the feature map $\phi(\,\cdot\,)$
and computing
$\langle \phi(x), \phi(x') \rangle_{\mathscr{H}_{\mathcal{X}}}$
explicitly
can be computationally expensive,
the inner product can be computed using $K_{\mathcal{X}}(x, x')$ directly
for a $K_{\mathcal{X}}$ that is positive definite.
This is known as the \emph{kernel trick} \cite{shawe2004kernel}.

By choosing $K_{\mathcal{X}}$, we effectively choose a basis
to represent the functions in $\mathscr{H}_{\mathcal{X}}$.  With the reproducing property, we can then write
the function $f$ as $f(x) = w^{\top} \phi(x)$, a weighted sum of basis functions for some
possibly infinite-dimensional weight vector $w$.   We wish to solve for the particular $w$ which,
based on the samples $\mathcal S$, minimizes the difference between the observations and the kernel-based estimate.


Let $\mathscr{P}$ denote the set of all probability measures on $\mathcal{X}$.
The \emph{kernel distribution embedding}
\cite{berlinet2011reproducing, smola2007hilbert}
of a probability measure
$\Pr_{\boldsymbol{y}\,|\,x, u} \in \mathscr{P}$,
given by
$\mu_{(x, u)} : \mathscr{P} \rightarrow \mathscr{H}_{\mathcal{X}}$,
is defined as
\begin{align}
	\mu_{(x, u)} \triangleq
	\int_{\mathcal{X}}
	K_{\mathcal{X}}(y, \cdot)
	\Pr_{\boldsymbol{y}\,|\,x, u}
	\lbrace \mathrm{d} y \,|\, x, u \rbrace
\end{align}

Let $\mathscr{H}_{\mathcal{X}}$ denote
the unique RKHS
for the state space $\mathcal{X}$
with the positive definite kernel
$K_{\mathcal{X}} : \mathcal{X} \times{} \mathcal{X} \rightarrow \Re$.
Similarly, let $\mathscr{H}_{\mathcal{X} \times \mathcal{U}}$ denote
the RKHS
for $\mathcal{X} \times \mathcal{U}$
with the positive definite kernel
$K_{\mathcal{X} \times \mathcal{U}} : (\mathcal{X}, \mathcal{U}) \times{} (\mathcal{X}, \mathcal{U}) \rightarrow \Re$.
We define the \emph{conditional distribution embedding} of the stochastic kernel
$Q \in \mathscr{P}$ as $\mu_{(x, u)}$.
Then, the expectation of $f$ with respect to the probability measure $Q$ is given by
\begin{align}
	\langle \mu_{(x,u)}, f \rangle_{\mathscr{H}_{\mathcal{X}}} =
	\mathbb{E}_{\boldsymbol{y} \sim Q(\,\cdot\,|\,x, u)}
	[ f(\boldsymbol{y}) ]
	\label{eqn: conditional distribution embedding inner product}
\end{align}
This means we can evaluate the expectation of a function
with respect to $Q$ as an inner product
in $\mathscr{H}_{\mathcal{X}}$.

We can construct an estimate $\bar{\mu}_{(x, u)}$ of $\mu_{(x, u)}$
\cite{smola2007hilbert}
from samples $\mathcal{S}$
to approximate
\eqref{eqn: conditional distribution embedding inner product},
\begin{align}
	\langle \bar{\mu}_{(x, u)}, f \rangle_{\mathscr{H}_{\mathcal{X}}} \approx
	\mathbb{E}_{\boldsymbol{y} \sim Q(\,\cdot \,|\, x, u)}
	\left[
		f(\boldsymbol{y})
	\right]
	\label{eqn: conditional distribution embedding inner product estimate}
\end{align}
According to the Riesz representation theorem
\cite{micchelli2005learning},
the element $\bar{\mu}_{(x, u)}$
can be viewed as the solution to a regularized least-squares problem
that minimizes the error of the expectation operator over the samples \cite{micchelli2005learning, lever2012conditional},
\begin{align}
    \min_{\bar{\mu}} \left\lbrace
		\frac{1}{M}
    \sum_{i=1}^{M} \lVert K_{\mathcal{X}}(\bar{y}_{i}, \cdot) - \bar{\mu}_{(\bar{x}_{i}, \bar{u}_{i})} \rVert_{\mathscr{H}_{\mathcal{X}}}^{2} + \lambda \lVert \bar{\mu} \rVert_{\mathscr{H}_{\Upsilon}}^{2} \right\rbrace
		\label{eqn: regularized least-squares}
\end{align}
where $\mathscr{H}_{\Upsilon}$ is a vector-valued RKHS \cite{micchelli2005learning}.
The solution is unique and has the form
\begin{align}
	\bar{\mu}_{(x, u)} =
	\eta
	\Phi^{\top}
	(G + \lambda M I)^{-1}
	\Psi
	K_{\mathcal{X} \times \mathcal{U}}((x, u), \cdot)
	\label{eqn: conditional distribution embedding estimate}
\end{align}
where $\lambda$ is a regularization parameter to avoid overfitting,
$\eta$ is a normalizing constant,
and $\Phi$, $\Psi$, and $G$ are defined as
\begin{align}
	\Phi &= [
		K_{\mathcal{X}}(\bar{y}_{1}, \cdot),
		\ldots,
		K_{\mathcal{X}}(\bar{y}_{M}, \cdot)
	]^{\top}
	\label{eqn: phi vector} \\
	\Psi &= [
		K_{\mathcal{X} \times \mathcal{U}}((\bar{x}_{1}, \bar{u}_{1}), \cdot),
		\ldots,
		K_{\mathcal{X} \times \mathcal{U}}((\bar{x}_{M}, \bar{u}_{M}), \cdot)
	]^{\top}
	\label{eqn: psi vector} \\
	G &= \Psi \Psi^{\top}
	\label{eqn: gram matrix}
\end{align}
%
By the reproducing property
of $K_{\mathcal{X}}$ in $\mathscr{H}_{\mathcal{X}}$,
$\forall f \in \mathscr{H}_{\mathcal{X}}$,
we can rewrite
\eqref{eqn: conditional distribution embedding inner product estimate} as
\begin{align}
  \langle \bar{\mu}_{(x, u)}, f \rangle_{\mathscr{H}_{\mathcal{X}}} &=
		\boldsymbol{f}^{\top} \beta(x, u)
	\label{eqn: conditional distribution embedding inner product estimate as sum}
\end{align}
where $\boldsymbol{f} = [f(\bar{y}_{1}), \ldots, f(\bar{y}_{M})]^{\top}$,
and $\beta(x, u) \in \Re^{M}$ is a vector of coefficients,
\begin{align}
	\beta(x, u) &=
	\eta
	(G + \lambda M I)^{-1}
	\Psi
	K_{\mathcal{X} \times \mathcal{U}}((x, u), \cdot)
	\label{eqn: compute beta}
\end{align}
\emph{
This means an approximation of the value function expectation
$\mathbb{E}_{\boldsymbol{y} \sim Q(\,\cdot \,|\, x, \pi_{k}(x))} \left[ V{}_{k+1}^{\pi}(\boldsymbol{y}) \right]$
in \eqref{eqn: terminal-hitting value k}
can be evaluated as a linear operation
in $\mathscr{H}_{\mathcal{X}}$.
}


\subsection{Terminal-Hitting Time Problem}
\label{section: terminal-hitting time problem}

With the conditional distribution embedding $\mu_{(x, u)}$,
the value functions in \eqref{eqn: terminal-hitting value k}
can be written as
\begin{align}
	V{}_{k}^{\pi}(x) =
	\boldsymbol{1}_{\mathcal{K}}(x)
	\langle
		\mu_{(x, \pi_{k}(x))}, V{}_{k+1}^{\pi}
	\rangle_{\mathscr{H}_{\mathcal{X}}}
	\label{eqn: terminal-hitting value k as inner product}
\end{align}
With the estimate $\bar{\mu}_{(x, u)}$
\eqref{eqn: conditional distribution embedding estimate},
we define the approximate value functions
$\overbar{V}_{k}^{\pi} : \mathcal{X} \rightarrow [0, 1]$, $k \in [0, N-1]$, as
\begin{align}
	\overbar{V}_{k}^{\pi}(x) =
	\boldsymbol{1}_{\mathcal{K}}(x)
	\langle
		\bar{\mu}_{(x, \pi_{k}(x))}, V{}_{k+1}^{\pi}
	\rangle_{\mathscr{H}_{\mathcal{X}}}
	\label{eqn: terminal-hitting approximate value k as inner product}
\end{align}
such that $\overbar{V}_{k}^{\pi}(x) \approx	V{}_{k}^{\pi}(x)$.
Thus, an approximation for the reach-avoid probability
$r_{x_{0}}^{\pi}(\mathcal{K}, \mathcal{T})$
can be computed via the backward recursion
described in \algorithmref{algo: backward recursion},
such that
\begin{align}
	\overbar{V}_{0}^{\pi}(x)
	\approx r_{x_{0}}^{\pi}(\mathcal{K}, \mathcal{T})
	\label{eqn: terminal-hitting approximate value 0}
\end{align}

\begin{algorithm}[!t]
	\caption{Value Function Estimate}
	\label{algo: backward recursion}
	\textbf{Input}:
	samples $\mathcal{S}$ drawn i.i.d. from $Q$,
	policy $\pi$,
	horizon $N$
	\\
	\textbf{Output}:
	value function estimate $\overbar{V}_{0}^{\pi}(x)$

	\begin{algorithmic}[1]
		\State Compute $\Phi$ and $\Psi$ using $\mathcal{S}$ from
		\eqref{eqn: phi vector} and
		\eqref{eqn: psi vector}
		\State $G \gets \Psi \Psi^{\top}$ 
		\State Initialize $\overbar{V}_{N}^{\pi}(x) \gets \boldsymbol{1}_{\mathcal{T}}(x)$
		\For{$k \gets N-1$ to $0$}
			\State Compute $\beta(x, \pi_{k}(x))$ from \eqref{eqn: compute beta}
			\State $
            \mathcal{Y}
			\gets
			[\overbar{V}_{k+1}^{\pi}(\bar{y}_{1}), \ldots, \overbar{V}_{k+1}^{\pi}(\bar{y}_{M})]^{\top}$
			\State
				$
					\overbar{V}_{k}^{\pi}(x) \gets
						\boldsymbol{1}_{\mathcal{K}}(x)
				\mathcal{Y}^{\top}
							\beta(x, \pi_{k}(x))
				$
		\EndFor
		\State Return $\overbar{V}_{0}^{\pi}(x) \approx r_{x_{0}}^{\pi}(\mathcal{K}, \mathcal{T})$
	\end{algorithmic}

\end{algorithm}



We now seek to characterize
the quality of the approximation
and the conditions for its convergence.
As in \cite{sriperumbudur2010hilbert, gretton2007kernel, gretton2012kernel},
we define a pseudometric
that characterizes the accuracy of the estimate $\bar{\mu}_{(x,u)}$.
\begin{definition}[Distance Pseudometric]
The distance pseudometric in $\mathscr{H}_{\mathcal{X}}$ between
the conditional distribution embedding
$\mu_{(x, u)} \in \mathscr{H}_{\mathcal{X}}$
and the estimate
$\bar{\mu}_{(x, u)} \in \mathscr{H}_{\mathcal{X}}$ is defined as
$\lVert \mu_{(x, u)} - \bar{\mu}_{(x, u)} \rVert_{\mathscr{H}_{\mathcal{X}}}$.
\end{definition}

It is shown in \cite{fukumizu2008kernel}
that if $K_{\mathcal{X}}$ is a \emph{characteristic}, bounded kernel,
then
$\lVert \mu_{(x, u)} - \bar{\mu}_{(x, u)} \rVert_{\mathscr{H}_{\mathcal{X}}} = 0$
if and only if $\mu_{(x, u)} = \bar{\mu}_{(x, u)}$.
A kernel is characteristic if the kernel embedding is injective,
meaning the embeddings for any two different conditional distributions
are represented by different elements
within the RKHS.
Thus, as
$\lVert \mu_{(x, u)} - \bar{\mu}_{(x, u)} \rVert_{\mathscr{H}_{\mathcal{X}}}$
converges
\cite{grunewalder2012modelling, smola2007hilbert},
the estimate converges in probability
to the conditional distribution embedding
within $\mathscr{H}_{\mathcal{X}}$.

\begin{lemma}\cite[Lemma~2.2]{grunewalder2012modelling}
	\label{lemma: conditional distribution embedding convergence}
	For any $\varepsilon > 0$,
	if the regularization parameter $\lambda$ in
	\eqref{eqn: regularized least-squares}
	is chosen such that
	$\lambda \rightarrow 0$
	and $\lambda^{3} M \rightarrow \infty$, and if $\mathcal{X}$ is bounded and $K_{\mathcal{X}}$ is strictly positive definite, then
	\begin{align}
		\Pr_{\mathcal{S} \sim Q}
		\left\lbrace
			\sup_{(x, u) \in \mathcal{X} \times \mathcal{U}}
			\lVert
				\mu_{(x, u)} - \bar{\mu}_{(x, u)}
			\rVert_{\mathscr{H}_{\mathcal{X}}}
			> \varepsilon
		\right\rbrace \rightarrow 0
		\label{eqn: conditional distribution embedding convergence}
	\end{align}
\end{lemma}

\begin{proposition}[Value Function Convergence]
	\label{prop: value function expectation convergence}
	For any $\varepsilon > 0$,
	if the regularization prameter $\lambda$ in
	\eqref{eqn: regularized least-squares}
	is chosen such that
	$\lambda \rightarrow 0$
	and $\lambda^{3} M \rightarrow \infty$, and if $\mathcal{X}$ is bounded and $K_{\mathcal{X}}$ is strictly positive definite,
	$\vert V{}_{k}^{\pi}(x) - \overbar{V}_{k}^{\pi}(x) \vert$
	converges in probability.
\end{proposition}

\begin{proof}
	By subtracting
	\eqref{eqn: terminal-hitting approximate value k as inner product} from
	\eqref{eqn: terminal-hitting value k as inner product}, and using the parallelogram law, we define
	the absolute value function error $\mathcal{E}_{k}(x)$ at time $k$,
	\begin{align}
		\mathcal{E}_{k}(x) &\triangleq
		\vert
			V{}_{k}^{\pi}(x) - \overbar{V}_{k}^{\pi}(x)
		\vert \\
&= \boldsymbol{1}_{\mathcal{K}}(x)
		\vert
			\langle
				\mu_{(x,\pi_{k}(x))} - \bar{\mu}_{(x,\pi_{k}(x))}, V{}_{k+1}^{\pi}
			\rangle_{\mathscr{H}_{\mathcal{X}}}
		\vert
		\label{eqn: value function absolute difference}
	\end{align}
	We can rewrite
	\eqref{eqn: value function absolute difference}
	using Cauchy-Schwarz to obtain
	\begin{align}
		\mathcal{E}_{k}(x)
		&\leq \boldsymbol{1}_{\mathcal{K}}(x)
		\lVert
			V{}_{k+1}^{\pi}
		\rVert_{\mathscr{H}_{\mathcal{X}}}
		\lVert
			\mu_{(x,\pi_{k}(x))} - \bar{\mu}_{(x,\pi_{k}(x))}
		\rVert_{\mathscr{H}_{\mathcal{X}}}
	\end{align}
	Since
	$\lVert \mu_{(x,\pi_{k}(x))} - \bar{\mu}_{(x,\pi_{k}(x))} \rVert_{\mathscr{H}_{\mathcal{X}}}$
	converges in probability
	according to
	\lemmaref{lemma: conditional distribution embedding convergence},
	$\vert V{}_{k}^{\pi}(x) - \overbar{V}_{k}^{\pi}(x) \vert$
	also converges in probability
	with the probabilistic error bound $\varepsilon$.
\end{proof}
Using this, the value function approximation in
\eqref{eqn: terminal-hitting approximate value k as inner product}
converges in probability
for some probabilistic error bound $\varepsilon$
as the number of samples increases.
%

\begin{corollary}
	\label{cor: value function expectation convergence cor1}
	For any $\varepsilon > 0$,
	the error in the reach-avoid probability
	computed using \algorithmref{algo: backward recursion}
	converges in probability to
	\begin{align}
		\vert
			V{}_{0}^{\pi}(x) - \overbar{V}_{0}^{\pi}(x)
		\vert \leq N\varepsilon
	\end{align}
\end{corollary}
\begin{proof}
	By subtracting
	\eqref{eqn: terminal-hitting approximate value k as inner product} from
	\eqref{eqn: terminal-hitting value k as inner product}, we obtain
	the absolute value function error $\mathcal{E}_{N-1}(x)$ at time $k = N-1$,
	\begin{align}
		\mathcal{E}_{N-1}(x) =
		\vert
			V{}_{N-1}^{\pi}(x) - \overbar{V}_{N-1}^{\pi}(x)
		\vert
		\label{eqn: appendix cor1 value function k error}
	\end{align}
	Using \propositionref{prop: value function expectation convergence},
    if the error in the approximate value function converges in probability to at most $\varepsilon$,
	then the error in \eqref{eqn: appendix cor1 value function k error} converges in probability to $\varepsilon$,
	i.e. $\mathcal{E}_{N-1}(x) \leq \varepsilon$.

	Because the error
	in the approximate value function
	for $k = N-1$
	converges in probability to $\varepsilon$,
	then by approximating and recursively substituting
	$\overbar{V}_{k}^{\pi}(x)$ for $k < N-1$,
	the error at time $k$ converges in probability to $(N-k)\varepsilon$.
	Thus, by induction, the error
	obtained by the backward recursion in
	\algorithmref{algo: backward recursion}
	converges in probability to at most $N\varepsilon$,
	\begin{align}
		\vert
			V{}_{0}^{\pi}(x) - \overbar{V}_{0}^{\pi}(x)
		\vert \leq N\varepsilon
	\end{align}
	which concludes the proof.
\end{proof}



\subsection{Maximal Reach-Avoid Policy in the Terminal Sense}

As in \eqref{eqn: terminal-hitting value k as inner product},
we write the optimal value functions $V{}_{k}^{*}$
from \eqref{eqn: terminal-hitting optimal value k}
using the conditional distribution embedding $\mu_{(x, u)}$.
\begin{align}
	V{}_{k}^{*}(x) =
	\sup_{u \in \mathcal{U}}
	\left\lbrace
		\boldsymbol{1}_{\mathcal{K}}(x)
		\langle
			\mu_{(x, u)}, V{}_{k+1}^{*}
		\rangle_{\mathscr{H}_{\mathcal{X}}}
	\right\rbrace
	\label{eqn: terminal-hitting optimal value k as inner product}
\end{align}
With the estimate $\bar{\mu}_{(x, u)}$
from \eqref{eqn: conditional distribution embedding estimate},
we define the approximate optimal value functions
$\overbar{V}_{k}^{*} : \mathcal{X} \rightarrow [0, 1]$, $k \in [0, N-1]$
\begin{align}
	\overbar{V}_{k}^{*}(x) =
	\sup_{u \in \mathcal{U}}
	\left\lbrace
	\boldsymbol{1}_{\mathcal{K}}(x)
	\langle
		\bar{\mu}_{(x, u)}, V{}_{k+1}^{*}
	\rangle_{\mathscr{H}_{\mathcal{X}}}
	\right\rbrace
	\label{eqn: terminal-hitting approximate optimal value k as inner product}
\end{align}
such that $\overbar{V}_{k}^{*}(x) \approx V{}_{k}^{*}(x)$.
%
%
If
$\bar{\pi}_{k}^{*} : \mathcal{X} \rightarrow \mathcal{U}$
is such that
\begin{align}
	\bar{\pi}_{k}^{*}(x) =
	\mathrm{arg}\sup_{u \in \mathcal{U}}
	\left\lbrace
		\boldsymbol{1}_{\mathcal{K}}(x)
		\langle
			\bar{\mu}_{(x, u)}, V{}_{k+1}^{*}
		\rangle_{\mathscr{H}_{\mathcal{X}}}
	\right\rbrace
	\label{eqn: approximate maximal reach-avoid policy in the terminal sense}
\end{align}
then
$\bar{\pi}^{*} = \lbrace \bar{\pi}_{0}^{*}, \bar{\pi}_{1}^{*}, \ldots \rbrace$
is the approximate maximal reach-avoid policy in the terminal sense.
The approximate optimal reach-avoid probability
under policy $\bar{\pi}^{*}$
initialized with $\overbar{V}_{k}^{*}(x) = \boldsymbol{1}_{\mathcal{T}}(x)$,
is described in
\algorithmref{algo: backward recursion} as
\begin{align}
  r_{x_{0}}^{*}(\mathcal{K}, \mathcal{T}) \approx
	\overbar{V}_{0}^{*}(x).
	\label{eqn: terminal-hitting approximate optimal value 0}
\end{align}


\section{Numerical Results}
\label{section: numerical results}

\begin{table*}[ht!]
	\caption{Computation Time and Parameters}
	\label{table: computation time}
	\centering
	\normalsize
	\renewcommand{\arraystretch}{1.1}
	\begin{tabular*}{\textwidth}{
		@{\extracolsep{\fill}} l
		@{\hspace{0.5\tabcolsep}} c
		@{\hspace{0.5\tabcolsep}} c
		@{\hspace{0.5\tabcolsep}} c
		@{\hspace{0.5\tabcolsep}} c
		@{\hspace{0.5\tabcolsep}} c
		}
		\toprule
		& Number of
		& Number of
		&  
		& Dynamic 
		& Chance-Constrained 
		\\
		  System
		& Samples [$M$]
		& Evaluation Points
		& \algorithmref{algo: backward recursion} 
		& Programming 
		& Open \cite{vinod2019affine} 
		\\
		\midrule
		Double Integrator
		& $1024$
		& $10201$
		& $0.43$ s
		& $31.76$ s
		& $24.36$ s
		\\
		CWH
		& $883$
		& $10201$
		& $0.48$ s
		& --
		& $34.51$ s
		\\
		\midrule
		10000-D Integrator
		& $1024$
		& $1$
		& $30.53$ s
		& --
		& --
		\\
		\bottomrule
	\end{tabular*}
\end{table*}

%

We implemented \algorithmref{algo: backward recursion} on two
well-known stochastic systems for the purpose of validation and error analysis.
We first generated $M = 1024$ samples via simulation, 
then presumed no knowledge of the system dynamics
or the stochastic disturbance in computing
$r_{x_{0}}^{\pi}(\mathcal{K}, \mathcal{T})$ via Algorithm \ref{algo: backward recursion}.
%
%
We chose a 
Gaussian radial basis function kernel,
 $K(x, x') = \exp ( - \lVert x - x' \rVert_{2}^{2} / 2\sigma^2 )$,
with $\sigma = 0.1$,
and regularization parameter $\lambda = 1$.
All computations were done in Matlab on a
2.7GHz Intel Core i7 CPU with 16 GB RAM.
Code to generate all figures is available at
\textit{https://github.com/unm-hscl/ajthor-LCSS-2019}.


\subsection{$n$-D Stochastic Chain of Integrators}

We consider a $n$-D stochastic chain of integrators
\cite{vinod2017scalable}, in which the input appears at the $n^{\mathrm{th}}$ derivative
and each element of the state vector is the discretized integral of the element that follows it. The dynamics with sampling time $T$ are given by:
\begin{align}
    \boldsymbol{x}_{k+1} =
    \begin{bmatrix}
      1 & T & \cdots & \frac{T^{n-1}}{(n-1)!} \\
      0 & 1 & \cdots & \frac{T^{n-2}}{(n-2)!} \\
      \vdots & \vdots & \ddots & \vdots \\
      0 & 0 & 0 & 1
    \end{bmatrix}
    \boldsymbol{x}_{k} +
    \begin{bmatrix}
      \frac{T^{n}}{n!} \\
      \frac{T^{n-1}}{(n-1)!} \\
      \vdots \\
      T
    \end{bmatrix}
    u_{k} +
    \boldsymbol{w}_{k}
  \end{align}
with i.i.d. disturbance $\boldsymbol{w}_{k}$
defined on the probability space
$(\mathcal{W}, \mathscr{B}(\mathcal{W}), \Pr_{\boldsymbol{w}})$.
We consider two distributions: 
%
1) a Gaussian distribution with variance $\Sigma = 0.01I$
such that $\boldsymbol{w}_{k} \sim \mathcal{N}(0, \Sigma)$, and
2) a beta distribution such that $\boldsymbol{w}_{k} \sim \textnormal{Beta}(\alpha, \beta)$, with PDF
	$f(x \,|\, \alpha, \beta) =
	\frac{
		\Gamma(\alpha + \beta)
	}{
		\Gamma(\alpha) \Gamma(\beta)
	}
	x^{\alpha-1} (1-x)^{\beta-1}$,
described in terms of Gamma function $\Gamma$ and
positive shape parameters $\alpha = 0.5$, $\beta = 0.5$.
%
The control policy $\pi$ is
$\pi_{0}(x) = \pi_{1}(x) = \ldots = 0$.
The target set and safe set are $\mathcal{T}, \mathcal{K} = [-1, 1]^{n}$.

For a system with $n=2$, we compute
approximate safety probabilities
using \algorithmref{algo: backward recursion}
for time horizon $N=3$ with the Gaussian distribution in Fig. \ref{fig: main_figure_1}(a) and
the beta
distribution in Fig. \ref{fig: main_figure_2}(a).
We compared the RKHS solution for the Gaussian distribution with a dynamic programming solution
implemented in \cite{vinod2019sreachtools}.
The absolute error (\ref{eqn: value function absolute difference})
(Fig. \ref{fig: main_figure_1}(b)) has a maximum value of 0.074.
We consider the region strictly within $\mathcal{K}$ to account for Matlab rounding errors.
The error is highest along the ridges on the upper right and lower left corners.
Fig. \ref{fig: dimensionality_time} shows that as the number of samples increases, the error decreases, as expected.
%

%

For a 10,000-dimensional system, which is beyond the
computational capabilities of any existing methods for stochastic reachability,
computation of \eqref{eqn: terminal-hitting approximate value 0} took 30.53 seconds for $x_0 = 0$ (Table \ref{table: computation time}).
Computation time, evaluated from the
the same initial condition for all systems of dimension 2 through 10,000 (Fig. \ref{fig: dimensionality_time}),
%
appears to increase linearly
because computation of the norm in the Gaussian kernel function
scales linearly with state dimension.
Note that the structure of the system dynamics plays no role in the computational complexity,
as the structure of $G$ in Algorithm~\ref{algo: backward recursion} does not depend on the dynamics.

\begin{figure}[t!]
    \centering
    \includegraphics{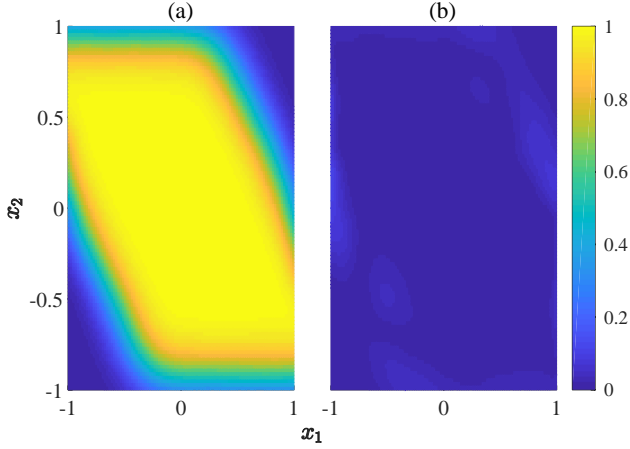}
    \caption{
    (a) Approximate safety probabilities computed using \algorithmref{algo: backward recursion} for a double integrator at $k=0$ for $N=3$.
	(b) Absolute error $\vert V_{0}^{\pi}(x) - \overbar{V}_{0}^{\pi}(x) \vert$
	between the dynamic programming solution
	and \algorithmref{algo: backward recursion}
	at $k=0$ for $N=3$.}
    \label{fig: main_figure_1}
\end{figure}

\begin{figure}[!t]
    \centering
	\includegraphics{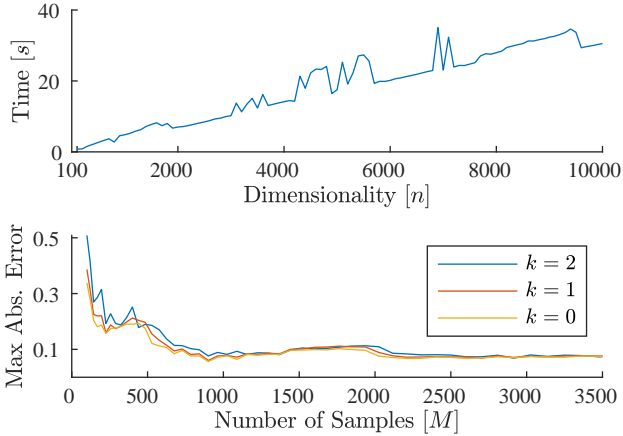}
	\caption{
		(top) System dimensionality [$n$] vs. average computation time [$s$] for an $n$-D stochastic integrator system.
		(bottom) Number of samples [$M$] vs. maximum absolute error $\vert V_{0}^{\pi}(x) - \overbar{V}_{0}^{\pi}(x) \vert$ for time steps $k = 2$ to $k = 0$ for $N = 3$.
	}
	\label{fig: dimensionality_time}
\end{figure}



\subsection{Clohessy-Wiltshire-Hill System}

Lastly, we considered the more realistic example of
spacecraft rendezvous and docking, in which
a spacecraft must dock with another spacecraft
while remaining within a line of sight cone.
The Clohessey-Wiltshire-Hill dynamics,
%
\begin{align}
		\ddot{x} - 3\omega x - 2 \omega \dot{y} &= F_{x}/m_{d} &
		\ddot{y} + 2 \omega \dot{x} &= F_{y}/m_{d}
\end{align}
with state $\boldsymbol{z} = [x, y, \dot{x}, \dot{y}] \in \mathcal{X} \subseteq \Re^{4}$,
 input
 $u = [F_{x}, F_{y}] \in \mathcal{U} \subseteq \Re^{2}$,
where $\mathcal{U} = [-0.1, 0.1] \times [-0.1, 0.1]$, and parameters $\omega$, $m_d$
can be written as a discrete-time LTI system
$\boldsymbol{z}_{k+1} = A \boldsymbol{z}_{k} + B u_{k} + \boldsymbol{w}_{k}$
with an additive Gaussian disturbance \cite{lesser2013stochastic}
%
%
with variance $\Sigma = \textnormal{diag}(1 \times 10^{-4}, 1 \times 10^{-4}, 5 \times 10^{-8}, 5 \times 10^{-8})$
such that $\boldsymbol{w}_{k} \sim \mathcal{N}(0, \Sigma)$.
The target set and safe set are defined
as in \cite{lesser2013stochastic}:
\begin{align}
	\mathcal{T} &=
	\lbrace
		\boldsymbol{z} \in \Re^{4} :
		\begin{aligned}[t]
				& \vert \boldsymbol{z}_{1} \vert \leq 0.1,
					-0.1 < \boldsymbol{z}_{2} < 0, \\
				& \vert \boldsymbol{z}_{3} \vert \leq 0.01,
					\vert \boldsymbol{z}_{4} \vert \leq 0.01 \rbrace
		\end{aligned} \\
	\mathcal{K} &=
	\lbrace
		\boldsymbol{z} \in \Re^{4} 	: \vert \boldsymbol{z}_{1} \vert < \vert \boldsymbol{z}_{2} \vert,
											\vert \boldsymbol{z}_{3} \vert \leq 0.05,
											\vert \boldsymbol{z}_{4} \vert \leq 0.05
	\rbrace
\end{align}
We generate samples using \cite{vinod2019sreachtools}
with a chance-affine controller.

Fig. \ref{fig: main_figure_2}(b) shows the approximate safety probabilities for time horizon $N = 5$, with a precomputed safety controller from \cite{vinod2019affine, vinod2019sreachtools}.
The safety probabilities for the entire region were computed in 0.48 seconds (Table \ref{table: computation time}), almost two orders of magnitude less than the chance constrained approach \cite{vinod2019affine}, which computes the set of initial conditions
where the safety probability is above a certain threshold (0.8 in this case).

\begin{figure}[t!]
    \centering
    \includegraphics{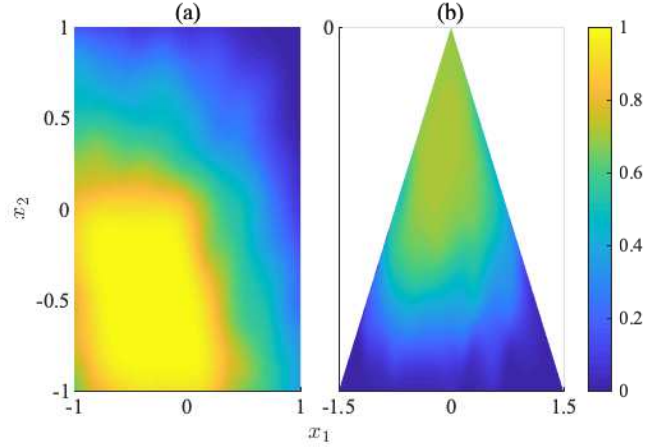}
    \caption{(a) Approximate safety probabilities computed using \algorithmref{algo: backward recursion} for a double integrator
	with a
	$\textnormal{Beta}(0.5, 0.5)$
	disturbance at $k=0$ for $N=1$.
	(b) Terminal-hitting safety probabilities for a CWH system
	with a Gaussian disturbance
	and a chance-affine controller
	at $k=0$ for $N=5$.}
    \label{fig: main_figure_2}
\end{figure}


\subsection{Sample Size and Parameter Tuning}

The number of samples used to create the estimate $\bar{\mu}_{(x, u)}$
is the most significant computational bottleneck for
\algorithmref{algo: backward recursion} and is generally $\mathcal{O}(M^3)$.
%
As the number of samples
increases,
the absolute error in the approximate safety probabilities decreases.
However, methods have been developed recently to
reduce the computational complexity \cite{rahimi2008random, le2013fastfood, lever2012conditional}
to $\mathcal{O}(M \log M)$.
Additionally, for high-dimensional systems,
the number of samples needed to fully characterize the
system dynamics and disturbance
increases as the system dimensionality increases,
which is prohibitive for analysis over large regions of the state space.
However, due to the sample-based nature of
\algorithmref{algo: backward recursion}, we can choose samples
within a local region of interest
in order to approximate the safety probabilities.

The kernel bandwidth parameter and the regularization parameter
are tunable parameters which can affect
the quality of the estimate obtained
using \algorithmref{algo: backward recursion}.
A cross-validation scheme to empirically
choose these parameters is presented in
\cite[Section~6]{micchelli2005learning}.

\section{Conclusions \& Future Work}

We present a sample-based method to compute the stochastic reachability probability measure
for Markov control systems
with arbitrary disturbances that does not require a known model of the transition kernel.
Our approach employs efficient computation associated with a reproducing kernel Hilbert space
to approximate conditional distributions via simple matrix operations.
The method is demonstrated on a 10,000-dimensional integrator as well as a realistic model of
relative spacecraft motion.
We plan to extend this
to sample-based controller synthesis, with application to
systems with autonomous and human elements.


\bibliographystyle{IEEEtran}
\bibliography{bibliography}


\end{document}